\documentclass[a4paper,10pt]{amsart}
\usepackage[colorlinks,linkcolor=blue,citecolor=blue]{hyperref}
\usepackage{latexsym, amssymb, amsmath, amsthm, mathrsfs, bbm}
\usepackage[all, knot]{xy}
\xyoption{arc}

\usepackage{anysize}\marginsize{23mm}{23mm}{25mm}{25mm}

\def \To{\longrightarrow}
\def \dim{\operatorname{dim}}
\def \gr{\operatorname{gr}}
\def \ord{\operatorname{ord}}

\def \C{\mathcal{C}}
\def \md{\mathcal{D}}
\def \D{\Delta}
\def \d{\delta}
\def \e{\varepsilon}

\def \R{\mathcal{GR}}

\def \Z{\mathbb{Z}}

\def \k{\mathbbm{k}}
\def \1{\mathbf{1}}
\def \Id{\operatorname{Id}}
\def \rep{\operatorname{rep}}
\def \ord{\operatorname{ord}}

\numberwithin{equation}{section}

\newtheorem{theorem}{Theorem}[section]
\newtheorem{lemma}[theorem]{Lemma}
\newtheorem{proposition}[theorem]{Proposition}
\newtheorem{corollary}[theorem]{Corollary}
\newtheorem{definition}[theorem]{Definition}

\newtheorem{remark}[theorem]{Remark}

\begin{document}

\title[THE GREEN RINGS OF MINIMAL HOPF QUIVERS]{THE GREEN RINGS OF MINIMAL HOPF QUIVERS$^\dag$}\thanks{$^\dag$Supported by PCSIRT IRT1264, SRFDP 20130131110001 and SDNSF ZR2013AM022.}

\subjclass[2010]{19A22, 18D10, 16G20}

\keywords{Green ring, tensor category, quiver representation}

\author[H.-L. Huang]{Hua-Lin Huang}
\address{School of Mathematics, Shandong University,
Jinan 250100, China} \email{hualin@sdu.edu.cn}

\author[Y. Yang]{Yuping Yang*}\thanks{*Corresponding author.}
\address{School of Mathematics, Shandong University,
Jinan 250100, China} \email{yupingyang.sdu@gmail.com}

\date{}
\maketitle

\begin{abstract}
Let $\k$ be a field and $Q$ a minimal Hopf quiver, i.e., a cyclic quiver or the infinite linear quiver, and let $\rep^{ln}(Q)$ denote the category of locally nilpotent finite dimensional $\k$-representations of $Q.$ The category $\rep^{ln}(Q)$ has natural tensor structures induced from graded Hopf structures on the path coalgebra $\k Q.$ Tensor categories of the form $\rep^{ln} (Q)$ are an interesting class of tame hereditary pointed tensor categories which are not finite. The aim of this paper is to compute the Clebsch-Gordan formulae and Green rings of such tensor categories.
\end{abstract}

\section{Introduction}
Throughout the paper, we work over an algebraically closed field $\k$ of characteristic $0.$ Vector spaces, (co)algebras, (co)modules, Hopf algebras, categories, morphisms and unadorned $\otimes$ are over $\k.$ By a tensor category is meant a locally finite abelian rigid monoidal category in which the neutral object is simple, see \cite{egno} for unexplained notions of tensor categories. Recall that the Green ring of a tensor category $\C,$ denoted by $\R(\C),$ is the free abelian group generated by the isomorphism classes $[X]$ of objects in $\C,$ with multiplication given by tensor product $[X] \cdot [Y] = [X \otimes Y]$ modulo all split short exact sequences. It is well known that Green rings are a convenient way of organizing information about direct sums and tensor products of tensor categories.

Given a tensor category $\C,$ it is certainly interesting to determine its Green ring $\R(\C).$ However, this mission is generally too complicated to be accomplished. Recently the Green rings were computed for some relatively less complicated tensor categories, for example the module categories of Taft algebras in \cite{cvoz}, the module categories of generalized Taft algebras in \cite{lz}, and pointed tensor categories of finite type \cite{gr1}. A key feature of these tensor categories investigated in \cite{cvoz, lz, gr1} is that there are only finitely many indecomposable objects, up to isomorphism, in them. In retrospect, this is the main reason why their Green rings are computable.

The aim of this paper is to compute the Green rings of some tame hereditary pointed tensor categories. In the viewpoint of the representation theory of algebras (see e.g. \cite{ass}), naturally this is a further question we may ask ourselves immediately after \cite{gr1}. A tensor category $\C$ is said to be pointed, if every simple object of $\C$ is invertible. By reconstruction theorem \cite{egno}, a pointed tensor category with a fiber functor can be presented as the category of finite dimensional right comodules over a pointed Hopf algebra. On the other hand, the tame hereditary condition is equivalent to saying that such a category is equivalent to the category of locally nilpotent finite dimensional representations of a cyclic quiver or the infinite linear quiver, see \cite{s, hsaq1}. Cyclic and infinite linear quivers are called minimal Hopf quivers as they are basic building blocks of general Hopf quivers \cite{cr2, hsaq1}. Note that finite dimensional indecomposable representations of cyclic and infinite linear quivers are explicitly classified in quiver representation theory \cite{ass} and the Hopf structures over such quivers are given in \cite{hsaq1}, now we are in a good position to compute the associated Green rings via similar idea of \cite{gr1}. We remark that the Green rings of the categories of quiver representations with the vertex-wise and arrow-wise tensor product were studied in \cite{h1,h2,k1,k2}. Note that such tensor structures are generally not induced from a bialgebra, and  hence are quite different from ours.

Here is the organization of the paper. In Section 2 we review some necessary facts. In Section 3 and Section 4, we compute the Clebsch-Gordan formulae and Green rings of the tensor categories associated to cyclic and infinite linear quivers respectively.

\section{Preliminaries}

In this section we recall some preliminary notions and facts about
quivers, representations, path coalgebras, Hopf quivers and tensor
categories.

\subsection{Quivers and path coalgebras}
A quiver is a quadruple $Q=(Q_0,Q_1,s,t),$ where $Q_0$ is the set of vertices,
$Q_1$ is the set of arrows, and $s,t: Q_1 \longrightarrow Q_0$ are
two maps assigning respectively the source and the target for each
arrow. For $a \in Q_1,$ we write $a:s(a) \To t(a).$ A vertex is, by
convention, said to be a trivial path of length $0.$ We also write
$s(g)=g=t(g)$ for each $g \in Q_0.$ The length of an arrow is set to
be $1.$ In general, a non-trivial path of length $n \ (\ge 1)$ is a
sequence of concatenated arrows of the form $p=a_n \cdots a_1$ with
$s(a_{i+1})=t(a_i)$ for $i=1, \cdots, n-1.$ By $Q_n$ we denote the
set of the paths of length $n.$

Let $Q$ be a quiver and $\k Q$ the associated path space which is
the $\k$-span of its paths. There is a natural coalgebra structure
on $\k Q$ with comultiplication as split of paths. Namely, for a
trivial path $g,$ set $\D(g)=g \otimes g$ and $\e(g)=1;$ for a
non-trivial path $p=a_n \cdots a_1,$ set
$$\D(p)=t(a_n) \otimes p + \sum_{i=1}^{n-1}a_n \cdots a_{i+1}
\otimes a_i \cdots a_1 +p \otimes s(a_1)$$ and $\e(p)=0.$ This is
the so-called path coalgebra of the quiver $Q.$

There exists on $\k Q$ an intuitive length gradation $\k
Q=\bigoplus_{n \geqslant 0}\k Q_n$ which is compatible with the
comultiplication $\D$ just defined. It is clear that the path
coalgebra $\k Q$ is pointed, and the set of group-like elements
$G(\k Q)$ is $Q_0.$  Moreover, the coradical filtration of $\k Q$ is
$$ \k Q_0 \subseteq \k Q_0 \oplus \k Q_1 \subseteq \k Q_0 \oplus \k Q_1
\oplus \k Q_2 \subseteq \cdots \ ,$$ therefore it is coradically
graded.

\subsection{Hopf quivers}
A quiver $Q$ is said to be a Hopf quiver if the corresponding path coalgebra $kQ$ admits a
graded Hopf algebra structure, see \cite{cr2}. Hopf quivers can be determined by
ramification data of groups. Let $G$ be a group, $\C$ the set of
conjugacy classes. A ramification datum $R$ of the group $G$ is a
formal sum $\sum_{C \in \C}R_CC$ of conjugacy classes with
coefficients in $\mathbb{N}=\{0,1,2,\cdots\}.$ The corresponding
Hopf quiver $Q=Q(G,R)$ is defined as follows: the set of vertices
$Q_0$ is $G,$ and for each $x \in G$ and $c \in C,$ there are $R_C$
arrows going from $x$ to $cx.$ For a given Hopf quiver $Q,$ the set
of graded Hopf structures on $kQ$ is in one-to-one correspondence
with the set of $kQ_0$-Hopf bimodule structures on $kQ_1.$

A Hopf quiver $Q = Q(G,R)$ is connected if and only if the union of the
conjugacy classes with non-zero coefficients in R generates G. We denote
the unit element of G by e. If $R_{\{e\}}\neq 0$, then there are $R_{\{e\}}$-loops attached
to each vertex; if the order of elements in a conjugacy class $C \neq {e}$ is $n$
and $R_C \neq 0$, then corresponding to these data in $Q$ there is a sub-quiver
$(n,R_C)$-cycle (called $n$-cyclic quiver if $R_C = 1$), i.e., the quiver having $n$
vertices, indexed by the set $\Z_n$ of integers modulo $n$, and $R_C$ arrows going from
$i$ to $i + 1$ for each $i \in \Z_n$; if the order of elements in a conjugacy class $C$ is $\infty$,
then in $Q$ there is a sub-quiver $R_C$-chain (called infinite linear quiver if $R_C = 1$),
i.e., a quiver having set of vertices indexed by the set $\Z$ of integral numbers,
and $R_C$ arrows going from $j$ to $j + 1$ for each $j \in \Z.$ Therefore, cyclic quivers
and the infinite linear quiver are basic building blocks of
general Hopf quivers and they are called minimal Hopf quivers.

\subsection{Quiver representations}
Let $Q$ be a quiver. A representation of $Q$ is a collection
$$V=(V_g,V_a)_{g \in Q_0, a \in Q_1}$$ consisting of a vector space
$V_g$ for each vertex $g$ and a linear map $V_a: V_{s(a)} \To
V_{t(a)}$ for each arrow $a.$ A morphism of representations $\phi: V
\To W$ is a collection $\phi=(\phi_g)_{g \in Q_0}$ of linear maps
$\phi_g:V_g \To W_g$ for each vertex $g$ such that
$W_a\phi_{s(a)}=\phi_{t(a)}V_a$ for each arrow $a.$ Given a
representation $V$ of $Q$ and a path $p,$ we define $V_p$ as
follows. If $p$ is trivial, say $p=g \in Q_0,$ then put
$V_p=\Id_{V_g}.$ For a non-trivial path $p=a_n \cdots a_2a_1,$ put
$V_p=V_{a_n} \cdots V_{a_2}V_{a_1}.$ A representation $V$ of $Q$ is
said to be locally nilpotent if for all $g \in Q_0$ and all $x \in
V_g,$ $V_p(x) = 0$ for all but finitely many paths $p$ with source $g.$ A representation $V$ is said to be finite dimensional, if $\sum_{g \in Q_0} \dim V_g < \infty.$ Let $\rep^{ln}(Q)$ denote the category of locally nilpotent finite dimensional representations of $Q.$ It is well known that the
category of finite dimensional right $\k Q$-comodules is equivalent to $\rep^{ln}(Q),$ see \cite{s}.

\subsection{Tensor categories}
A monoidal category is a sextuple
$(\C,\otimes,\1,\alpha,\lambda,\rho),$ where $\C$ is a category,
$\otimes: \C \times \C \to \C$ is a bifunctor, $\1$ an object (to be called neutral),
$\alpha: \otimes \circ ( \otimes \times \Id) \to \otimes \circ (\Id
\times \otimes), \lambda: \1 \otimes - \to \Id, \rho: - \otimes \1
\to \Id$ are natural isomorphisms such that the associativity and
unitarity constrains hold, or equivalently the pentagon and the
triangle diagrams are commutative. A tensor category is a locally finite abelian rigid monoidal category in which the neutral object is simple.
Natural examples of tensor categories are the categories of finite dimensional $H$-modules and
$H$-comodules where $H$ is a Hopf algebra equipped with an invertible antipode. Recall
that, if $U$ and $V$ are right $H$-comodules, let $U \otimes V$ be
the usual tensor product of $\k$-spaces and the comodule structure
is given by $u \otimes v \mapsto u_0 \otimes v_0 \otimes u_1v_1,$
where we use the Sweedler notation $u \mapsto u_0 \otimes u_1$ for
comodule structure maps. The neutral object is the trivial comodule
$\k$ with comodule structure map $k \mapsto k \otimes 1.$ On the
other hand, by the reconstruction formalism, tensor categories with
fiber functors are coming in this manner. For more details on tensor categories, see \cite{egno}.

\section{The Green ring of a cyclic quiver}

\subsection{Hopf structures over a cyclic quiver}
Let $G=\langle g | g^n=1\rangle$ be a cyclic group of order $n$ and let
$\mathcal{Z}$ denote the Hopf quiver $Q(G,g).$ The quiver
$\mathcal{Z}$ is a cyclic quiver of form $$ 
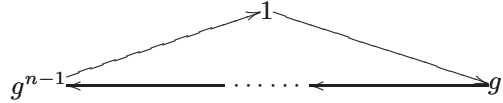
\xy {\ar
(0,0)*{1}; (30,-10)*{g}}; {\ar (-30,-10)*{g^{n-1}}; (0,0)*{1}}; {\ar
(30,-10)*{g}; (3,-10)*{ \cdots  \ }}; {\ar (-3,-10)*{\ \cdots };
(-30,-10)*{g^{n-1}}}
\endxy $$ If $n=1,$ then
$\mathcal{Z}$ is the one-loop quiver, that is, consisting of one
vertex and one loop. It is easy to see that such a quiver provides
only the familiar divided power Hopf algebra in one variable, which
is isomorphic to the polynomial algebra in one variable \cite{hsaq1}.

From now on we assume $n > 1.$ For each integer $i \in \Z_n,$ let $a_i$ denote the
arrow $g^i \longrightarrow g^{i+1}.$ Let $p_i^l$ denote the path
$a_{i+l-1} \cdots a_{i+1}a_i$ of length $l.$ Then $\{p_i^l \ | \ i \in \Z_n, \ l \ge 0 \}$ is a basis of $\k \mathcal{Z}.$ We also need some notations of Gaussian binomials. For
any $q \in \k,$ integers $l,m \ge 0,$ let
\[
l_q=1+ q + \cdots +q^{l-1}, \ \ l!_q=1_q \cdots l_q, \ \
\binom{l+m}{l}_q=\frac{(l+m)!_q}{l!_qm!_q}.
\]
When $1\ne q\in \k$ is an $n$-th root of unity of multiplicative order $d,$
\begin{equation}
{l+m \choose l}_q = 0 \ \ \text{if and only if} \ \
\big[\frac{l+m}{d} \big]-\big[ \frac{m}{d}\big]-\big[ \frac{l}{d}
\big]>0,
\end{equation}
where $[x]$ means the integer part of $x$.

Now we recall the graded Hopf structures on $\k \mathcal{Z}.$ By
\cite{cr2}, they are in one-to-one correspondence with the
$\k G$-module structures on $\k a_0,$ and in turn with the set of $n$-th
roots of unity. For each $q \in \k$ with $q^n=1,$ let $g.a_0=qa_0$
define a $\k G$-module. The corresponding $\k G$-Hopf bimodule is $\k G
\otimes_{\k G} \k a_0 \otimes \k G = \k a_0 \otimes \k G.$ We identify
$a_i=a_0 \otimes g^i.$ This is how we view $\k \mathcal{Z}_1$ as a
$\k G$-Hopf bimodule. The following path multiplication formula
\begin{equation} \label{e3.1}
p_i^l \cdot p_j^m = q^{im} {{l+m}\choose l} _q p_{i+j}^{l+m}
\end{equation}
was given in \cite{cr2} by induction. In particular,
\begin{equation} \label{e3.2}
g \cdot p_i^l=q^lp_{i+1}^l, \quad  p_i^l \cdot g=p_{i+1}^l, \quad
a_0^l=l_q!p_0^l \ .
\end{equation}
For each $q,$ the corresponding graded Hopf algebra is denoted by
$\k \mathcal{Z}(q).$ The following lemma gives the algebra structure by generators and relations.

\begin{lemma}\emph{(\cite[Lemma 3.2 ]{hsaq1}) }
As an algebra, $\k \mathcal{Z}(q)$ can be presented by generators and relations as follows:
\begin{itemize}
\item[(1)] If $q=1,$ then the generators are $g$ and $a_0$, and the relations are $g^n=1$ and $ga_0=a_0g.$
\item[(2)] If $\ord(q)=d>1,$ then the generators are $g, \ a_0$ and $p^d_0$, and the relations are $g^n=1, \ ga_0=qa_0g, \ a_0^d=0, \ a_0p_0^d=p_0^da_0$ and $gp^d_0=p_0^dg.$
\end{itemize}
\end{lemma}

\subsection{Tensor category associated to $\k \mathcal{Z}(q)$}
The aim of this section is to compute the Green ring of the tensor category of finite dimensional right $\k \mathcal{Z}(q)$-comodules. As mentioned in Subsection 2.3, as a category it is equivalent to the category of locally nilpotent finite dimensional representations of the quiver $\mathcal{Z}.$ For each $i \in \Z_n$ and integer $l \ge 0,$ let $V(i,l)$ be a vector space of dimension $l+1$ with a basis $\{v_m^i\}_{0 \leq m \leq l}.$  $V(i,l)$ is made into a representation of $\mathcal{Z}$ by putting $V(i,l)_j$ the $\k$-span of $\{v_m^i | i+m=j \ \mathrm{in} \ \Z_n\}$ and letting $V(i,l)_{a_j}$ maps $v_m^i$ to
$v^i_{m+1}$ if $i+m=j \ \mathrm{in} \ \Z_n.$ Here by convention $v^i_k$ is understood as $0$ if $k>l.$ Note that $V(i,l)$ is viewed as a $\k \mathcal{Z}(q)$-comodule by
\begin{eqnarray}
\d :V(i,l) & \longrightarrow & V(i,l)\otimes \k \mathcal{Z}(q) \notag \\
v^i_m & \mapsto & \sum^l_{j=m}v^i_j\otimes p_{i+m}^{j-m}.
\end{eqnarray}
Using (3.1), the comodule structure map of $V(i,l)\otimes V(j,m)$ is given by
\begin{equation}
\d(v^i_s\otimes v^j_t)=\sum^l_{x=s}\sum^m_{y=t}q^{(i+s)(y-t)}\binom{x+y-s-t}{x-s}_qv^i_x\otimes v^j_y\otimes p_{i+j+s+t}^{x+y-s-t}.
\end{equation}
It is well known that $\{V(i,l) | i \in \Z_n, \ l \ge 0\}$ is a complete set of indecomposable objects of $\rep^{ln}(\mathcal{Z}),$ hence a complete set of finite dimensional indecomposable $\k\mathcal{Z}(q)$-comodules.

For application in latter computations, we also view $V(i,l)\otimes V(j,m)$ as a rational module of $\gr(\k\mathcal{Z}(q))^*,$  the graded dual algebra of $\k\mathcal{Z}(q).$ The module structure map is given by
\begin{equation}
(p_e^f)^*(v^i_s\otimes v^j_t)=\delta_{e,i+j+s+t}\sum_{x+y=f}\sum_{x=0}^{l-s}\sum_{y=0}^{m-t}q^{(i+s)(f-x)}\binom{f}{x}_qv^i_{s+x}\otimes v^j_{t+y}.
\end{equation}
Note that $\gr(\k\mathcal{Z}(q))^*$ is actually the path algebra associated to $\mathcal{Z}(q),$ so we have
\begin{equation}
(p_e^f)^*(v_s^i\otimes v_t^j)=(p_{e+k}^{f-k})^*(p_e^k)^*(v_s^i\otimes v_t^j)
\end{equation}
for all $0\leq k\leq f.$

From now on we denote by $\C_q$ the tensor category of $\k\mathcal{Z}(q)$-comodules for brevity and by $\R(\C_q)$ its Green ring.

\subsection{The case of $q=1$}
As $q=1,$ the Gaussian binomial coefficients shrink to the usual ones, i.e., $\binom{f}{x}_1 = \binom{f}{x}$. We start with some useful lemmas.
\begin{lemma}
$V(1,0)\otimes V(i,l) = V(i+1,l) = V(i,l)\otimes V(1,0), \quad V(1,0)^{\otimes n}=V(0,0)$ in $\C_1.$
\end{lemma}
\begin{proof}
Define $F: V(1,0)\otimes V(i,l)\longrightarrow V(i+1,l)$ by $F(v^1_0\otimes v^i_s)=v^{i+1}_s.$ It is easy to verify that the map $F$ is an isomorphism in $\C_1.$ Similarly one can prove the remaining equalities.
\end{proof}

\begin{lemma}
\begin{equation}
V(0,1)\otimes V(0,l)=V(0,l+1)\oplus V(1,l-1)=V(0,l)\otimes V(0,1).
\end{equation}
\end{lemma}
\begin{proof}
 Define  maps
\begin{eqnarray*}
\phi_1 :V(0,l+1) & \longrightarrow  & V(0,1)\otimes V(0,l) \\
    v_0^0        & \mapsto          & v_0^0\otimes v_0^0\\
    v_i^0        & \mapsto          & v^0_0\otimes v^0_i+iv^0_1\otimes v^0_{i-1}\\
    v_{l+1}^0    & \mapsto          & (l+1)v^0_1\otimes v^0_l
\end{eqnarray*}
where $i=1,2, \cdots, l$ and
\begin{eqnarray*}
\phi_2:V(1,l-1) & \longrightarrow  & V(0,1)\otimes V(0,l) \\
   v^1_j        & \mapsto          & v^0_0\otimes v^0_{j+1}-(l-j)v^0_1\otimes v^0_{j}
\end{eqnarray*}
where $0\leq j\leq l-1.$
Now we verify that the two maps are comodule monomorphisms. Obviously the two maps are injective. Using (3.4) and the definition of $\phi_1$ we have
\begin{equation*}
\begin{split}
\delta(\phi_1(v_0^0))&=\delta(v_0^0\otimes v_0^0 )\\
&=\sum_{y=0}^{l}v_0^0\otimes v_y^0\otimes p_0^y +\sum_{y=0}^{l}(y+1)v_1^0\otimes v_y^0\otimes p_0^{y+1}\\
&=(\phi_1\otimes id)\delta(v_0^0),
\end{split}
\end{equation*}
\begin{equation*}
\begin{split}
\delta(\phi_1(v_j^0))&=\delta(v_0^0\otimes v_j^0 +jv_1^0\otimes v_{j-1}^0)\\
&=\sum_{x=j}^{l+1}\bigg[v_0^0\otimes v_x^0\otimes p_j^{x-j} +(x-j)v_1^0\otimes v_{x-1}^0\otimes p_j^{x-j}\bigg] + \sum_{y=j}^{l+1}jv_1^0\otimes v_{y-1}^0\otimes p_i^{y-j}\\
&=\sum_{x=j}^{l+1}\bigg[v_0^0\otimes v_x^0 + xv_1^0\otimes v_{x-1}^0\bigg]\otimes p_j^{x-j}\\
&=(\phi_1\otimes id)\delta(v_j^0)
\end{split}
\end{equation*}for $1\leq j\leq l$  and
\begin{equation*}
\delta(\phi_1(v_{l+1}^0))=\delta(v_1^0\otimes v_l^0)=v_1^0\otimes v_l^0\otimes p_{l+1}^0=(\phi_1\otimes id)\delta(v_{l+1}^0).
\end{equation*}
So we verified that $\phi_1$ is indeed a comodule monomorphism. Similarly one proves that $\phi_2$ is also a comodule monomorphism. Considering the indices of the basis of
$\phi_1(V(0,l+1))$ and $\phi_2(V(1,l-1))$, we can see that  $\phi_1(V(0,l+1))\cap \phi_2(V(1,l-1))=\{0\}.$
Then we obtain that $V(0,l+1)\oplus V(1,l-1)\cong \phi_1(V(0,l+1))\oplus \phi_2(V(1,l-1)) \subset V(0,1)\otimes V(0,l)$.  By comparing the dimensions we obtain the equality.

Define maps
\begin{eqnarray*}
\psi_1 :V(0,l+1) & \longrightarrow  & V(0,l )\otimes V(0,1) \\
    v_0^0        & \mapsto          & v_0^0\otimes v_0^0\\
    v^0_j        & \mapsto          & v^0_j\otimes v^0_0+jv^0_{j-1}\otimes v^0_1\\
    v_{l+1}^0    & \mapsto          & (l+1)v^0_l\otimes v^0_1
\end{eqnarray*}
for $0\leq j\leq l$ and
\begin{eqnarray*}
\psi_2:V(1,l-1) & \longrightarrow  & V(0,1)\otimes V(0,l) \\
   v^1_j        & \mapsto          & v^0_{j+1}\otimes v^0_0+(l-j)v^0_{j}\otimes v^0_1
\end{eqnarray*}
for $0\leq j\leq l-1.$
One can prove the second equality in a similar manner.
\end{proof}

For the convenience of later computations, now we make a convention. We will understand the identity in Lemma 3.3 as \[ V(0,l+1)=V(0,1)\otimes V(0,l) - V(1,l-1)\] thanks to the Krull-Schmidt theorem. Similar identities which interpret the direct sum decomposition of indecomposable objects by suitable subtraction will be used freely in the rest of the paper.

Now we are ready to decompose $V(0,l)\otimes V(0,m).$

\begin{theorem}
\begin{equation*}
V(0,l)\otimes V(0,m)=
\begin{cases} \bigoplus_{i=0}^lV(i,l+m-2i), & l\leq m; \\ \bigoplus_{i=0}^mV(i,l+m-2i), & l> m.
\end{cases}
\end{equation*}
\end{theorem}
\begin{proof}
We only prove the case with $l\leq m$ as  the other case is similar.
Induce on $l.$ For $l=1,$  the claim is Lemma 3.3. Now assume $l \ge 2$ and $V(0,j)\otimes V(0,m)= \bigoplus_{i=0}^jV(i,j+m-2i)$ if $j<l$.
Then by Lemma 3.3 we have
\begin{equation*}
\begin{split}
V(0,l)\otimes V(0,m)
 &=\big[V(0,1)\otimes V(0,l-1)-V(1,l-2)\big]\otimes V(0,m)\\
 &=V(0,1)\otimes \bigg[\bigoplus_{i=0}^{l-1}V(i,l+m-1-2i)\bigg]-\bigoplus_{i=0}^{l-2}V(1+i,l+m-2-2i)\\
 &=\bigoplus_{i=0}^lV(i,l+m-2i).
\end{split}
\end{equation*}

The proof is completed.
\end{proof}

In order to present the ring $\R(\C_1),$ we need the generalized Fibonacci polynomials used in \cite{cvoz}.
\begin{definition}
For integer $k \ge 0$, define the polynomial $f_k(x,y) \in \Z[x,y]$ inductively by
\begin{itemize}
\item[1)]$f_0(x,y)=1$ and $f_1(x,y)=y$,  \\
\item[2)]$f_k(x,y)=yf_{k-1}(x,y)-xf_{k-2}(x,y)$ if $k\geq 2.$
\end{itemize}
\end{definition}
The general formula for $f_k(x,y)$ is given by \cite[Lemma 3.11]{cvoz} as follows
\begin{equation}
f_k(x,y)=\sum_{i=0}^{[\frac{k}{2}]}(-1)^i\binom{k-i}{i}x^iy^{k-2i},
\end{equation}
where $[\frac{k}{2}]$ means the integer part of $\frac{k}{2}.$

We need the following fact about $f_k(x,y)$ whose proof is easy and hence omitted.
\begin{lemma}
$\{x^if_k(x,y)|k,i \ge 0\}$ is a $\Z$-basis of $\Z[x,y].$
\end{lemma}

Now we are in the position to present the Green ring $\R(\C_1).$
\begin{theorem}
 $\R(\C_1) \cong \Z[x,y]/\langle x^n-1\rangle.$
\end{theorem}
\begin{proof}
Define
\begin{eqnarray*}
\Phi : \Z[x,y] & \longrightarrow & \R(\C) \\
         x  & \mapsto & [V(1,0)], \\
         y  & \mapsto & [V(0,1)].
\end{eqnarray*}
Obviously $\Phi$ is surjective as $[V(1,0)]$ and $[V(0,1)]$ generate $\R(\C)$ by Lemmas 3.2-3.3 and Theorem 3.4.  By the definition of $f_k(x,y)$, Lemma 3.2 and Lemma 3.3 we have
\begin{equation*}
\Phi(x^if_k(x,y))=[V(1,0)]^{i}f_k([V(1,0)],[V(0,1)]) =[V(i,k)]. \quad \quad (*)
\end{equation*}
Note that $\Phi (x^n-1)=0$ as $V(1,0)^{\otimes n}=V(0,0)=\1.$ It follows that $J=\langle x^n-1\rangle \subset \ker\Phi,$ so $\Phi$ induces an epimorphism
\begin{eqnarray*}
\overline{\Phi} : \Z[x,y]/J & \longrightarrow & \R(\C) \\
         \overline{x}  & \mapsto & [V(1,0)], \\
         \overline{y}  & \mapsto & [V(0,1)].
\end{eqnarray*}
Now use Lemma 3.6, it is not hard to see that
$$\{\overline{x^if_k(x,y)}|{0\leq i\leq n-1, \ k \ge 0}\}$$ makes a $\Z$-basis of $\Z[x,y]/J,$ hence $\overline{\Phi}$ is an isomorphism by $(*)$.
\end{proof}

\subsection{The case of $q$ being a nontrivial root of unity}
In this subsection $q$ is set to be a root of unity with multiplicative order $d=\ord(q) \ge 2.$ Similar to Lemma 3.2, we have
\begin{lemma}
$V(1,0)\otimes V(i,l) = V(i+1,l) = V(i,l)\otimes V(1,0), \quad V(1,0)^n=V(0,0)$ in $\C_q.$
\end{lemma}
\begin{proof}
Define $F_1:V(1,0)\otimes V(i,l)\longrightarrow V(i+1,l)$ by $F_1(q^jv^1_0\otimes v^i_j)=v^{i+1}_j$ and $F_2:V(i,l)\otimes V(1,0)\longrightarrow V(i+1,l)$ by $F_2(v^1_0\otimes v^i_j)=v^{i+1}_j.$ It is easy to verify that $F_1$ and $F_2$ provide the desired equalities.
\end{proof}

The Clebsch-Gordan problem for $C_q$ is much more complicated than that of $\C_{1}.$ We split the problem into several cases.

\begin{lemma}
\begin{equation*}
\begin{split}
& \ V(0,1)\otimes V(0,md+l)\\
=&\begin{cases} V(0,md+l+1)\oplus V(1,md+l-1), & 0\leq l\leq d-2; \\ V(0,(m+1)d-1)\oplus V(1,(m+1)d-1),  &
l=d-1.
\end{cases}\\
=& \ V(0,md+l)\otimes V(0,1).
\end{split}
\end{equation*}

\end{lemma}
\begin{proof}
We only prove the first equality as the second one is similar. For the case with $0\leq l\leq d-2,$ the proof is essentially the same as that of Lemma 3.3 and so omitted. Now consider $l=d-1.$ In this case we define
\begin{eqnarray*}
\phi_1:V(0,md+d-1) & \longrightarrow & V(0,1)\otimes V(0,md+d-1) \\
         v_k^0  & \mapsto &v_0^0\otimes v_k^0 + k_q v_1^0\otimes v_{k-1}^0
\end{eqnarray*}
and
\begin{eqnarray*}
\phi_2:V(1,md+d-1) & \longrightarrow & V(0,1)\otimes V(0,md+d-1) \\
         v_k^1  & \mapsto & q^kv_1^0\otimes v_k^0
\end{eqnarray*} for $0\leq k\leq md+d-1.$
The verification of $\phi_1$ and $\phi_2$ being comodule monomorphisms is similar to Lemma 3.3. It is also  easy to see that $\phi_1(V(0,(m+1)d-1))\cap \phi_2(V(1,(m+1)d-1))=\{0\},$ so we have $$V(0,(m+1)d-1)\oplus V(1,(m+1)d-1)\subseteq V(0,1)\otimes V(0,md+l)$$ and the claimed equality is obtained by comparing the dimensions.
\end{proof}
\begin{lemma}
Assume $V=\oplus_{s,t}V(s,t)$ and $V(i,j) \subset V$ in $\C_q$. Then there exists an inclusion map $\phi:V(i,j) \To V(s,t)$ for some indecomposable direct summand $V(s,t)$ of $V$ with $t\geq j$ and $i+j=s+t$ in $\Z_n.$
\end{lemma}
\begin{proof}
Suppose $\psi:V(i,j) \rightarrow V$ is the comodule inclusion and $\psi(v_0^i)=\sum_{s,t}\sum_{l=0}^{t}a_l^{s,t}v_l^s.$ By (3.3) and the fact that $\psi$ is a comodule map, we have
\[ \delta \circ \psi(v_0^i) = \sum_{s,t}\sum_{l=0}^{t}a_l^{s,t}\sum_{m=l}^tv_m^s\otimes p_{s+l}^{m-l}=\sum_{k=0}^j\psi(v_k^i)\otimes p_i^k=(\psi\otimes id) \circ \delta(v_0^i). \]
Comparing the third tensor factors of the two terms in the middle, it follows that there must be some $s,t$ such that $p_{s+l}^{t-l}=p_i^j.$ This obviously leads to $t \geq j$ and $i+j=s+t$ in $\Z_n.$ For the $s,t$ we chose, it is not hard to verify that $\phi:V(i,j) \rightarrow V(s,t)$ with $\phi(v_k^i)=v_{t-j+k}^{s}$ for $0\leq k \leq j$ is a comodule inclusion. We are done.
\end{proof}

\begin{lemma}
For all $0<l\leq d-1,$ we have
\begin{equation*}
 V(0,l)\otimes V(0,md)=V(0,md+l)\oplus \bigoplus_{i=1}^lV(i,md-1)=V(0,md)\otimes V(0,l).
\end{equation*}
\end{lemma}
\begin{proof}
As before we only prove the first identity. By (3.6) we have
\begin{equation*}
(p_{i+j}^{md})^*(v_i^0\otimes v_j^0)=\sum_{x+y=md}\sum_{x=0}^{l-i}\sum_{y=0}^{md-j}q^{i(md-x)}\binom{md}{x}_qv_{i+x}^0\otimes v_{j+y}^0.
\end{equation*}
According to (3.1), $\binom{md}{x}_q\neq 0$ iff $x= kd$ for some $k$. Note that $x\leq l-i\leq l \leq d-1$, this implies $\binom{md}{x}_q\neq 0$ iff  $x=0.$  So we have $(p_{i+j}^{md})^*(v_i^0\otimes v_j^0)\neq 0$ iff $j=0.$ As
\begin{equation*}
(p_0^{md})^*(v_0^0\otimes v_0^0)=v_0^0\otimes v_{md}^0
\quad \text{and} \quad
(p_{md}^l)^*(v_0^0\otimes v_{md}^0)=v_l^0\otimes v_{md}^0,
\end{equation*}
it follows that $\{(p_0^k)^*(v_0^0\otimes v_0^0)\}_{0\leq k\leq md+l}$ spans a subcomodule isomorphic to $V(0,md+l)$.

If we assume $i \ge 0, \ j>0$ and $i+j \leq l,$ then $(p_{i+j}^{md})^*(v_i^0\otimes v_j^0)=0,$  while by (3.7)
\begin{equation*}
\begin{split}
(p_{i+j}^{md-1})^*(v_i^0\otimes v_j^0)
&=(p_{i+j+(m-1)d}^{d-1})^*\big[(p_{i+j}^{(m-1)d})^*(v_i^0\otimes v_j^0)\big]\\
 &=(p_{i+j+(m-1)d}^{d-1})^*(v_i^0\otimes v_{j+(m-1)d}^0)\\
 &=\sum_{x+y=d-1}\sum_{x=0}^{l-i}\sum_{y=0}^{d-j}q^{i(d-1-x)}\binom{d-1}{x}_qv_{i+x}^0\otimes v_{j+(m-1)d+y}^0\\
 &\neq 0.
\end{split}
\end{equation*}
 It follows that $\{(p_{i+j}^k)^*(v_i^0\otimes v_j^0)|0\leq k\leq md-1\}$ spans a subcomodule isomorphic to $V(i+j,md-1).$  Hence we have proved that $V(0,md+l)$ and $V(k,md-1)$ for $0<k<l$ are subcomodules of $V(0,l)\otimes V(0,md).$  Assume $V(0,l)\otimes V(0,md)=\bigoplus_{s,t} V(s,t),$ then by the previous lemma $V(0,md+l)$ or $V(k,md-1)$ must be included in some $V(s,t)$ with $md+l=s+t$ or $k+md-1=s+t$ in $\Z_n$. It is clear that $V(0,md+l)$ and $V(k,md-1)$ for $0<k<l$ are in different $V(s,t)$ and therefore $$V(0,l+md)\oplus \bigoplus_{i=1}^lV(i,md-1) \subseteq V(0,l)\otimes V(0,md).$$ Now the identity is obtained by comparing the dimensions.
\end{proof}

\begin{lemma}For all $m \geq 1$ we have
\begin{equation*}
\begin{split}
&\quad V(0,d)\otimes V(0,md) \\
&=V(0,(m+1)d)\oplus V(1,(m+1)d-2)\oplus \bigoplus_{i=2}^{d-1}V(i,md-1)\oplus V(d,(m-1)d)\\
&=V(0,md)\otimes V(0,d).
\end{split}
\end{equation*}
\end{lemma}
\begin{proof}
We only prove the first identity. First we verify that each indecomposable summand of the second formula is included in the first term.

It is obvious that $V(0,(m+1)d)\subset V(0,d)\otimes V(0,md)$ since $$(p_0^{(m+1)d})^*(v_0^0\otimes v_0^0)=mv_d^0\otimes v_{md}^0,\ \ \ \ (p_0^{(m+1)d+1})^*(v_0^0\otimes v_0^0)=0.$$
Further, by (3.7) we have
\begin{equation*}
\begin{split}
(p_{1}^{(m+1)d-1})^*(v_1^0\otimes v_0^0-\frac{1}{m}v_0^0\otimes v_1^0)
 &=(p_{md+1}^{d-1})^*\big[(p_{1}^{md})^*(v_1^0\otimes v_0^0-\frac{1}{m}v_0^0\otimes v_1^0)\big]\\
 &=(p_{md+1}^{d-1})^*(v_1^0\otimes v_{md}^0-v_d^0\otimes v_{(m-1)d+1}^0)\\
 &=v_d^0\otimes v_{md}^0-v_d^0\otimes v_{md}^0\\
 &=0
\end{split}
\end{equation*}
and
\begin{equation*}
\begin{split}
(p_{1}^{(m+1)d-2})^*(v_1^0\otimes v_0^0-\frac{1}{m}v_0^0\otimes v_1^0)
 &=(p_{md+1}^{d-2})^*\big[(p_{1}^{md})^*(v_1^0\otimes v_0^0-\frac{1}{m}v_0^0\otimes v_1^0)\big]\\
 &=(p_{md+1}^{d-2})^*(v_1^0\otimes v_{md}^0-v_d^0\otimes v_{(m-1)d+1}^0)\\
 &=v_{d-1}^0\otimes v_{md}^0-v_d^0\otimes v_{md-1}^0\\
 &\neq 0,
\end{split}
\end{equation*}
hence $V(1,(m+1)d-2) \subset V(0,d)\otimes V(0,md).$

Now for all $i>0, \ j>0$ with $i+j<d,$ one may verify that $\{(p_{i+j}^k)^*(v_i^0\otimes v_j^0)|0\leq k\leq md-1\}$ spans a subcomodule isomorphic to $V(i+j,md-1)$ as the preceding lemma.

Next we prove that $V(d,(m-1)d)$ is a subcomodule of $V(0,d) \otimes V(0,md).$  Set $\alpha_i=(-1)^iq^{-\frac{(1+i)i}{2}}$ for $0\leq i\leq d,$ then one has $(p_{md}^1)^*(\sum_{i=0}^d\alpha_i v_i^0\otimes v_{md-i}^0)=0$ by direct computation. Since we have
\[(p_d^{(m-1)d})^*(v_0^0\otimes v_d^0)=(m-1)v_d^0\otimes v_{(m-1)d}^0+v_0^0\otimes v_{(m)d}^0, \quad
(p_d^{(m-1)d})^*(v_d^0\otimes v_0^0)=v_d^0\otimes v_{(m-1)d}^0 ,\]
and \[ (p_d^{(m-1)d})^*(v_i^0\otimes v_{d-i}^0)=v_i^0\otimes v_{md-i}^0, \quad i\neq 0,d. \]
It follows that
\begin{equation*}
(p_{d}^{(m-1)d})^*\bigg(\sum_{i=0}^d\alpha_i v_i^0\otimes v_{d-i}^0-(m-1)v_d^0\otimes v_0^0\bigg)=\sum_{i=0}^d\alpha_i v_i^0\otimes v_{md-i}^0\neq 0
\end{equation*}
and
\begin{equation*}
(p_{d}^{(m-1)d+1})^*\bigg(\sum_{i=0}^d\alpha_i v_i^0\otimes v_{d-i}^0-(m-1)v_d^0\otimes v_0^0\bigg)=p_{md}^1\bigg(\sum_{i=0}^d\alpha_i v_i^0\otimes v_{md-i}^0\bigg)
=0.
\end{equation*}
This implies that $\{(p_d^k)^*(\sum_{i=0}^d\alpha_i v_i^0\otimes v_{d-i}^0-(m-1)v_d^0\otimes v_0^0)|0\leq k\leq (m-1)d\}$ spans a subcomodule isomorphic to $V(d,(m-1)d)$ and hence $V(d,(m-1)d)$ is a subcomodule of $V(0,d)\otimes V(0,md).$

Again, using  Lemma 3.10, one can show that $$V(0,(m+1)d)\oplus V(1,(m+1)d-2)\oplus \bigoplus_{i=2}^{d-1}V(i,md-1)\oplus V(d,(m-1)d) \subseteq V(0,d)\otimes V(0,md)$$ and the identity is obtained by comparing the dimensions.
\end{proof}

\begin{corollary}
The Green ring $\R(\C_q)$ is commutative and is generated by $[V(1,0)], \ [V(0,1)]$ and $[V(0,d)].$
\end{corollary}
\begin{proof}
Direct consequence of Lemmas 3.8-3.9 and Lemmas 3.11-3.12 .
\end{proof}

\begin{remark}
Thanks to Lemmas 3.8-3.9 and Lemmas 3.11-3.12, we have
\begin{equation*}
\begin{split}
V(i,l)\otimes V(j,m)&=V(j,m )\otimes V(i,l)\\
&=V(i+j,0)\otimes V(0,l)\otimes V(0,m)\\
&=V(i+j,0)\otimes V(0,m)\otimes V(0,l).
\end{split}
\end{equation*}
Hence it is enough to compute the decomposition of $V(0,l)\otimes V(0,m)$ with $l\leq m.$
\end{remark}

\begin{lemma}
For $0\leq l\leq m \leq d-1,$ set $\gamma= l+m-d+1$ and we have
\begin{equation*}
V(0,l)\otimes V(0,m)=
\begin{cases}\bigoplus_{i=0}^lV(i,l+m-2i),& l+m\leq d-1;\\ \bigoplus_{i=0}^{\gamma}V(i,d-1)\oplus \bigoplus_{j=\gamma+1}^lV(j,l+m-2j), & l+m>d-1.
\end{cases}
\end{equation*}
\end{lemma}
\begin{proof}
For the case with $l+m\leq d-1,$ the proof is similar to that of Theorem 3.4 and so omitted. Now let $l+m\geq d$ and we will prove the lemma by induction on $l$. If $l=1,$ then by assumption $m=d-1,$ and in this case the claim has been proved in Lemma 3.8. If $l>1$ and $l+m=d,$ then we have
\begin{equation*}
\begin{split}
&\quad V(0,l)\otimes V(0,m)\\
&=\big[V(0,1)\otimes V(0,l-1)-V(1,l-2)\big]\otimes V(0,m)\\
&=\bigoplus_{i=0}^{l-1}\bigg[V(i,l+m-2i)\oplus V(i+1,l-2+m-2i)\bigg]-\bigoplus_{j=0}^{l-2}V(i+1,l-2+m-2j)\\
&=V(0,d-1)\oplus V(1,d-1) \oplus \bigoplus_{j=2}^lV(j,l+m-2j),
\end{split}
\end{equation*}
where in the second equality the case with $l+m \le d-1$ is applied. Similarly one can prove the case with $l>1$ and $l+m=d+1.$ Now let $l+m>d+1$ and by the induction hypothesis we have
\begin{equation*}
\begin{split}
&\quad V(0,l)\otimes V(0,m)\\
&=\big[V(0,1)\otimes V(0,l-1)-V(1,l-2)\big]\otimes V(0,m)\\
&=V(0,1)\otimes\bigg[\bigoplus_{i=0}^{\gamma-1}V(i,d-1)\oplus \bigoplus_{j=\gamma}^lV(j,l-1+m-2j)\bigg] \\ & \ \ \ -\bigg[\bigoplus_{i=0}^{\gamma-2}V(i+1,d-1)\oplus \bigoplus_{j=\gamma-1}^lV(j+1,l-2+m-2j)\bigg]\\
&=\bigoplus_{i=0}^{\gamma}V(i,d-1)\oplus \bigoplus_{j=\gamma+1}^lV(j,l+m-2j).
\end{split}
\end{equation*}
The proof is finished.
\end{proof}

\begin{lemma}
\begin{equation*}
\begin{split}
&V(0,l)\otimes V(0,md+h)\\
=&\begin{cases}\bigoplus_{i=0}^l V(i,md+h+l-2i),& 0< l\leq h<d-1,\ l+h\leq d-1;\\ \bigoplus_{i=0}^hV(i,md+h+l-2i)\oplus \bigoplus_{j=h+1}^lV(j,md-1),&  0< h< l<d-1,\ l+h\leq d-1;\\ \bigoplus_{i=0}^{\gamma}V(i,(m+1)d-1)\oplus \bigoplus_{j=\gamma+1}^lV(j,md+h+l-2j),& 0< l\leq h\leq d-1, \ l+h\geq d;\\ \bigoplus_{i=0}^{\gamma}V(i,(m+1)d-1)\oplus \bigoplus_{j=\gamma+1}^hV(j,md+h+l-2j) & \\ \oplus \bigoplus_{k=h+1}^lV(k,md-1),& 0< h< l\leq d-1, \ l+h\geq d.
\end{cases}
\end{split}
\end{equation*}
\end{lemma}

\begin{proof}
The case with $0< l\leq h<d-1, \ l+h\leq d-1$ can be proved in the same manner as Theorem 3.4. The proofs for the remaining three cases are similar, so we only provide the proof for the case with $0< h< l<d-1, \ l+h\leq d-1.$

If $l-h=1$, we have
\begin{equation*}
\begin{split}
&\quad V(0,l)\otimes V(0,md+h)\\
=&\big[V(0,1)\otimes V(0,l-1)-V(1,l-2)\big]\otimes V(0,md+h)\\
=&\bigoplus_{i=0}^{l-1}\bigg[V(i,md+h+l-2i)\oplus V(i+1,md+h+l-2-2i)\bigg]-\bigoplus_{i=0}^{l-2}V(i+1,md+h+l-2-2i)\\
=&\bigoplus_{i=0}^hV(i,md+h+l-2i)\oplus V(l,md-1).
\end{split}
\end{equation*}
Similarly one can prove the formula for the case with $l=h+2.$ Next assume $l-h>2.$ In this situation we have
\begin{equation*}
\begin{split}
&\quad V(0,l)\otimes V(0,md+h)\\
=&\big[V(0,1)\otimes V(0,l-1)-V(1,l-2)\big]\otimes V(0,md+h)\\
=&\bigoplus_{i=0}^h\bigg[V(i,md+h+l-2i)\oplus V(i+1,md+h+l-2-2i)\bigg]\oplus \bigoplus_{j=h+1}^{l-1}\bigg[V(j,md-1)\oplus V(j+1,md-1)\bigg] \\
&\ \  -\bigg[\bigoplus_{i=0}^hV(i+1,md+h+l-2-2i)\oplus \bigoplus_{j=h+1}^{l-2}V(j+1,md-1)\bigg]\\
=&\bigoplus_{i=0}^hV(i,md+h+l-2i)\oplus \bigoplus_{j=h+1}^lV(j,md-1).
\end{split}
\end{equation*}

We are done.
\end{proof}

\begin{lemma}
For all $h>0$ we have
\begin{equation*}
\begin{split}
&V(0,d)\otimes V(0,md+h)\\
=&\begin{cases}V(0,(m+1)d+h)\oplus  \bigoplus_{i=1}^{h}V(i,(m+1)d-1)\oplus V(h+1,,(m+1)d-h-2)& \\
\oplus \bigoplus_{j=h+2}^{d-1}V(j,md-1)\oplus  V(d,(m-1)d+h),& h \leq d-2; \\
V(0,(m+1)d+d-1)\oplus\bigoplus_{i=1}^{d-1}V(i,(m+1)d-1)\oplus V(d,md-1),& h=d-1.
\end{cases}
\end{split}
\end{equation*}
\end{lemma}
\begin{proof}
Only the proof for the case with $h<d-2$ is provided as the proofs for other cases are similar and much easier.
 We prove by induction on $m$. When $m=1$ we have
 \begin{eqnarray*}
& &V(0,d)\otimes V(0,d+h) \\
&=& V(0,d)\otimes \bigg[V(0,h)\otimes V(0,d)-\bigoplus_{i=1}^hV(i,d-1)\bigg]\\
&=& V(0,h)\otimes \bigg[V(0,2d)\oplus V(1,2d-2)\oplus \bigoplus_{i=2}^{d-1}V(i,d-1) \oplus V(d,0)\bigg]-V(0,d)\otimes \bigg[\bigoplus_{i=0}^hV(i,d-1)\bigg]\\
&=& V(0,2d+h)\oplus \bigoplus_{i=1}^hV(i,2d-1)\oplus \bigoplus_{i=0}^{h-1}V(1+i,2d-1)\oplus V(h+1,2d-2-h) \\
& &\oplus  \bigoplus_{i=2}^{d-1}\bigoplus_{j=0}^hV(i+j,d-1) \oplus V(d,h)-\bigoplus_{i=1}^h\bigg[V(i,2d-1) \oplus \bigoplus_{j=1}^{d-1}V(i+j,d-1)\bigg]\\
&=&V(0,2d+h)\oplus \bigoplus_{i=1}^{h}V(i,2d-1)\oplus  V(h+1,,2d-h-2) \oplus\bigoplus_{j=h+2}^{d-1}V(j,d-1)\oplus V(d,h).
\end{eqnarray*}
Note that for the first equality Lemma 3.11 is applied, for the second Lemma 3.12 and for the third Lemma 3.11. Similarly one can verify the formula for $m=2.$ In the following computation we assume $m\geq 3.$
\begin{eqnarray*}
&&V(0,d)\otimes V(0,md+h) \\
&=& V(0,d)\otimes \bigg[V(0,h)\otimes V(0,md)-\bigoplus_{i=1}^hV(i,md-1)\bigg]\\
&=& V(0,h)\otimes \bigg[V(0,(m+1)d)\oplus V(1,(m+1)d-2)\oplus\bigoplus_{i=2}^{d-1}V(i,md-1) \oplus V(d,(m-1)d)\bigg] \\
& &- V(0,d)\otimes \bigg[\bigoplus_{i=0}^hV(i,md-1)\bigg]\\
&=& V(0,(m+1)d+h)\oplus \bigoplus_{i=1}^hV(i,(m+1)d-1)\oplus \bigoplus_{i=0}^{h-1}V(1+i,(m+1)d-1) \\
&&\oplus V(h+1,(m+1)d-2-h)\oplus \bigoplus_{i=2}^{d-1}\bigoplus_{j=0}^hV(i+j,md-1)\oplus V(d,(m-1)d+h)\\
& &\oplus \bigoplus_{i=1}^hV(i+d,(m-1)d-1)\\
&&-\bigoplus_{i=1}^h\bigg[ V(i,(m+1)d-1)\oplus \bigoplus_{j=1}^{d-1}V(i+j,md-1) \oplus V(i+d,(m-1)d-1)\bigg] \\
&=&V(0,(m+1)d+h)\oplus  \bigoplus_{i=1}^{h}V(i,(m+1)d-1)\oplus  V(h+1,,(e+1)d-h-2) \\
&& \oplus\bigoplus_{j=h+2}^{d-1}V(j,ed-1)\oplus V(d,(m-1)d+h).
\end{eqnarray*} We note that in the third equality the inductive assumption is applied.

The proof is completed.
\end{proof}

Now we are ready to give the Clebsh-Gordan formulae for the tensor category $\C_q.$ The aim is to decompose $V(0,ed+f)\otimes V(0,md+h)$ into direct sum of indecomposable summands for all $e, \ m$ and $0\leq f, \ h\leq d-1$. According to Remark 3.14, in the following we may assume $ed+f \leq md+h.$ In this situation, $e\leq m,$ and if $f>h$ then $e<m.$ Set $\gamma=f+h-d+1.$
\begin{theorem}
\begin{itemize}
\item[(1)] If $0\leq f\leq h$ and $f+h<d-1,$ then
\begin{equation*}
\begin{split}
&V(0,ed+f)\otimes V(0,md+h)=\bigoplus_{k=0}^{e-1}\bigg[\bigoplus_{i=0}^{f}V(kd +i,(e+m-2k)d+f+h-2i) \\
&\quad \oplus \bigoplus_{j=f+1}^{h}V(kd+j,(e+m-2k)d-1) \oplus   \bigoplus_{r=h+1}^{f+h+1}V(kd+r,(e+m-2k)d+h+f-2r) \\
&\quad  \oplus \bigoplus_{s=f+h+2}^{d-1}V(kd+s,(e+m-1-2k)d-1)\bigg] \oplus \bigoplus_{i=0}^{f}V(ed+i,(m-e)d+h+f-2i);
\end{split}
\end{equation*}
\item[(2)]if $0<f\leq h<d$ and $f+h\geq d,$ then
\begin{equation*}
\begin{split}
&V(0,ed+f)\otimes V(0,md+h)=\bigoplus_{k=0}^{e-1}\bigg[\bigoplus_{i=0}^{\gamma}V(kd+i,(e+m-2k)d+d-1) \\
&\quad \oplus \bigoplus_{j=\gamma+1}^{f}V(kd+j,(e+m-2k)d+f+h-2i) \oplus   \bigoplus_{r=f+1}^{h}V(kd+r,(e+m-2k)d-1)\\  &\quad \oplus \bigoplus_{s=h+1}^{d-1}V(kd+s,(e+m-2k)d+h+f-2s)\bigg] \oplus \bigoplus_{i=0}^{\gamma}V(ed+i,(m-e)d+d-1) \\
&\quad \oplus \bigoplus_{i=\gamma+1}^{f}V(ed+i,(m-e)d+f+h-2i);
\end{split}
\end{equation*}
\item[(3)]if $0\leq h<f<d$ and $f+h < d-1,$ then
\begin{equation*}
\begin{split}
&V(0,ed+f)\otimes V(0,md+h) \\
=& \bigoplus_{k=0}^{e-1}\bigg[\bigoplus_{i=0}^{h}V(kd+i,(e+m-2k)d+h+f-2i)\oplus\bigoplus_{j=h+1}^{f}V(kd+j,(e+m-2k)d-1)\\
&\oplus\bigoplus_{r=f+1}^{h+f+1}V(kd+r,(e+m-2k)d+h+f-2r)\oplus \bigoplus_{s=f+h+2}^{d-1}V(kd+s,(e+m-1-2k)d-1)\bigg] \\
&\oplus \bigoplus_{i=0}^{h}V(ed+i,(m-e)d+h+f-2i)\oplus \bigoplus_{i=h+1}^{f}V(ed+i,(m-e)d-1);
\end{split}
\end{equation*}
\item[(4)]if $0<h<f<d$ and $f+h\geq d,$ then
\begin{equation*}
\begin{split}
&V(0,ed+f)\otimes V(0,md+h)=\bigoplus_{k=0}^{e-1}\bigg[\bigoplus_{i=0}^{\gamma}V(kd+i,(e+m-2k)d+d-1) \\
&\quad \oplus \bigoplus_{j=\gamma+1}^{h}V(kd+j,(e+m-2k)d+f+h-2j) \oplus \bigoplus_{r=h+1}^{f}V(kd+r,(e+m-2k)d-1) \\
&\quad \oplus \bigoplus_{s=f+1}^{d-1}V(kd+s,(e+m-2k)d+h+f-2s)\bigg] \oplus \bigoplus_{i=0}^{\gamma}V(ed+i,(m-e)d+d-1) \\
&\quad \oplus \bigoplus_{i=\gamma+1}^{h}V(ed+i,(m-e)d+f+h-2i) \oplus  \bigoplus_{i=h+1}^{f}V(ed+i,(m-e)d-1).
\end{split}
\end{equation*}
\end{itemize}
\end{theorem}
\begin{proof}
The rules of the decomposition of $V(0,\alpha) \otimes V(0,\beta)$ are divided into four cases. We prove the theorem by induction on $\alpha.$ The claim has been proved for the situation with $ 1\leq \alpha \leq d,$ thanks to Lemmas 3.11, 3.12, 3.16 and 3.17. Now assume that the claim holds when $1 \leq \alpha \leq (e-1)d,$ and we will prove that it still does when $ (e-1)d+1\leq \alpha \leq ed.$

Because the proofs of the four cases are similar, so we only focus on case (1) and the proofs of other cases are omitted.

If $f=1,$ then
\[ V(0,(e-1)d+1)\otimes V(0,md+h) =[V(0,1)\otimes V(0,(e-1)d)-V(1,(e-1)d-1)]\otimes V(0,md+h). \]
By the inductive assumption of case (1) we have
\begin{equation*}
\begin{split}
&V(0,(e-1)d)\otimes V(0,md+h) \\
=&\bigoplus_{k=0}^{e-2}\bigg[V(kd,(e-1+m-2k)d+h)\oplus \bigoplus_{i=1}^{h}V(kd+i,(e-1+m-2k)d-1)\\
& \oplus  V(kd+h+1,(e-1+m-2k)d-h-2) \oplus\bigoplus_{j=h+2}^{d-1}V(kd+j,(e+m-2-2k)d-1)\bigg] \\
& \oplus  V((e-1)d,(m-e+1)d+h).
\end{split}
\end{equation*}
It follows that
\begin{equation*}
\begin{split}
&V(0,1)\otimes V(0,(e-1)d)\otimes V(0,md+h)\\
=&\bigoplus_{k=0}^{e-2}\bigg\{V(kd,(e-1+m-2k)d+h+1)\oplus V(kd+1,(e-1+m-2k)d+h-1)\\
& \oplus \bigoplus_{i=1}^{h}\bigg[V(kd+i,(e-1+m-2k)d-1)\oplus V(kd+i+1,(e-1+m-2k)d-1)\bigg]\\
& \oplus V(kd+h+1,(e-1+m-2k)d-h-1)\oplus V(kd+h+2,(e-1+m-2k)d-h-3)\\
& \oplus\bigoplus_{j=h+2}^{d-1}\bigg[V(kd+j,(e+m-2-2k)d-1)\oplus V(kd+j+1,(e+m-2-2k)d-1)\bigg] \bigg\} \\
& \oplus  V((e-1)d,(m-e+1)d+h+1)\oplus V((e-1)d+1,(m-e+1)d+h-1).
\end{split}
\end{equation*}
By the inductive assumption of case (4) we have
\begin{equation*}
\begin{split}
&V(1,(e-1)d-1)\otimes V(0,md+h)\\
=&\bigoplus_{k=0}^{e-2}\bigg[\bigoplus_{i=0}^{h}V(kd+i+1,(e-1+m-2k)d-1)\oplus \bigoplus_{j=h+1}^{d-1}V(kd+j+1,(e-2+m-2k)d-1)\bigg].
\end{split}
\end{equation*}
 Subtracting the preceding two identities we obtain (1) with $\alpha=(e-1)d+1.$

Next assume $2\leq f\leq d-1.$ First note that
\begin{equation*}
\begin{split}
&V(0,(e-1)d+f)\otimes V(0,md+h)\\
=&V(0,1)\otimes V(0,(e-1)d+f-1)\otimes V(0,md+h)-V(1,(e-1)d+f-2)\otimes V(0,md+h).
\end{split}
\end{equation*}
Then by the inductive hypothesis of case (1), we have
\begin{equation*}
\begin{split}
&V(0,1)\otimes V(0,(e-1)d+f-1)\otimes V(0,md+h)\\
&=\bigoplus_{k=0}^{e-2}\bigg\{\bigoplus_{i=0}^{f-1}\bigg[V(kd+i,(e-1+m-2k)d+f+h-2i)\oplus V(kd+i+1,(e-1+m-2k)d+f-2+h-2i)\bigg]\\
&\ \ \ \oplus\bigoplus_{j=f}^{h}\bigg[V(kd+j,(e-1+m-2k)d-1)\oplus V(kd+j+1,(e-1+m-2k)d-1)\bigg]\\
&\ \ \ \oplus \bigoplus_{r=h+1}^{f+h}\bigg[V(kd+r,(e-1+m-2k)d+h+f-2r)\oplus V(kd+r+1,(e-1+m-2k)d+h+f-2r-2)\bigg]\\
&\ \ \ \oplus\bigoplus_{s=f+h+1}^{(d-1)}\bigg[V(kd+s,(e+m-2-2k)d-1)\oplus V(kd+s+1,(e+m-2-2k)d-1)\bigg] \bigg\} \\
&\ \ \ \oplus \bigoplus_{i=0}^{f-1}\bigg[V((e-1)d+i,(m-e+1)d+h+f-2i)\oplus V((e-1)d+i+1,(m-e+1)d+h+f-2-2i)\bigg].
\end{split}
\end{equation*}

Again applying the inductive hypothesis of case (1) we get
\begin{equation*}
\begin{split}
&V(1,(e-1)d+f-2)\otimes V(0,md+h)\\
=&\bigoplus_{k=0}^{e-2}\bigg[\bigoplus_{i=0}^{f-2}V(kd+i+1,(e-1+m-2k)d+f-2+h-2i)\oplus\bigoplus_{j=f-1}^{h}V(kd+j+1,(e-1+m-2k)d-1)\\
& \oplus\bigoplus_{r=h+1}^{f+h-1}V(kd+r+1,(e-1+m-2k)d+h+f-2-2r)\oplus \bigoplus_{s=f+h}^{d-1}V(kd+s+1,(e+m-2-2k)d-1)\bigg]\\
& \oplus \bigoplus_{i=0}^{f-2}V(i+1,(m-e+1)d+h+f-2-2i).
\end{split}
\end{equation*}
Now subtracting the foregoing two identities, we obtain (1) for $\alpha=(e-1)d+f$ with $1\leq f\leq d-1.$

Finally we prove (1) for $f=d,$ i.e. $\alpha=ed.$ In the  following we assume $2\leq h\leq d-2.$ The proof for the situation with $h=1$ or $h=d-1$ can be given in a similar and much easier way and thus we omit it.

By Lemma 3.12, we have
\begin{equation*}
\begin{split}
&V(0,ed)\otimes V(0,md+h)\\
=&\{V(0,d)\otimes V(0,(e-1)d)-[V(1,ed-2)\oplus \bigoplus_{i=2}^{d-1}V(i,(e-1)d-1)\oplus V(d,(e-2)d)]\}\otimes V(0,md+h).
\end{split}
\end{equation*} Then we apply the inductive assumption to get
\begin{equation*}
\begin{split}
 &V(0,(e-1)d)\otimes V(0,md+h)\\
=&\bigoplus_{k=0}^{e-2}\bigg[V(kd,(e-1+m-2k)d+h)\oplus \bigoplus_{i=1}^{h}V(kd+i,(e-1+m-2k)d-1) \\
&\oplus  V(kd+h+1,(e-1+m-2k)d-h-2)    \\
& \oplus \bigoplus_{j=h+2}^{d-1}V(kd+j,(e+m-2-2k)d-1)\bigg] \oplus  V((e-1)d,(m-e+1)d+h),
\end{split}
\end{equation*}
and apply Lemma 3.17 to get
\begin{equation*}
\begin{split}
&V(0,d)\otimes V(0,(e-1)d)\otimes V(0,md+h)\\
=&\bigoplus_{k=0}^{e-2}\bigg\{\bigg[V(kd,(e+m-2k)d+h)\oplus  \bigoplus_{i=1}^{h}V(kd+i,(e+m-2k)d-1)\\
&\oplus V(kd+h+1,,(e+m-2k)d-h-2)\oplus \bigoplus_{j=h+2}^{d-1}V(kd+j,(e+m-1-2k)d-1)\\
&\oplus  V((k+1)d,(e+m-2-2k)d+h)\bigg]\oplus \bigoplus_{i=1}^h\bigg[V(kd+i,(e+m-2k)d-1)\\
&\oplus\bigoplus_{j=1}^{d-1}V(kd+i+j,(e-1+m-2k)d-1)\oplus V((k+1)d+i,(e-2+m-2k)d-1)\bigg]\\
&\oplus \bigg[V(kd+h+1,(e+m-2k)d-h-2)\oplus  \bigoplus_{i=1}^{d-h-2}V(kd+h+1+i,(e+m-1-2k)d-1)\\
&\oplus V((k+1)d,(e+m-2-2k)d+h) \oplus \bigoplus_{j=d-h}^{d-1}V(kd+h+1+j,(e+m-2-2k)d-1)\\
&\oplus  V((k+1)d+h+1,(e+m-2-2k)d-h-2)\bigg]\oplus \bigoplus_{i=h+2}^{d-1}\bigg[V(kd+i,(e+m-1-2k)d-1)\\
&\oplus\bigoplus_{j=1}^{d-1}V(kd+i+j,(e+m-2-2k)d-1)\oplus V((k+1)d+i,(e+m-3-2k)d-1)\bigg] \bigg\}\\
&\oplus \bigg[V((e-1)d,(m-e+2)d+h)\oplus  \bigoplus_{i=1}^{h}V((e-1)d+i,(m-e+2)d-1)\\
&\oplus V((e-1)d+h+1,,(m-e+2)d-h-2)\oplus \bigoplus_{j=h+2}^{d-1}V((e-1)d+j,(m-e+1)d-1)\\
&\oplus  V(ed,(m-e)d+h)\bigg].
\end{split}
\end{equation*}
Next we apply the inductive assumption of formulae in (1) and (4) to get
\begin{equation*}
\begin{split}
&\bigg[V(1,ed-2)\oplus \bigoplus_{i=2}^{d-1}V(i,(e-1)d-1)\oplus V(d,(e-2)d)\bigg]\otimes V(0,md+h)\\
=&\quad\bigg\{\bigoplus_{k=0}^{e-2}\bigg[\bigoplus_{i=0}^{h-1}V(kd+i+1,(e+m-2k)d-1) \oplus V(kd+h+1,(e+m-2k)d-2-h)\\
& \oplus \bigoplus_{r=h+1}^{d-2}V(kd+r+1,(e-1+m-2k)d-1) \oplus V((k+1)d,(e+m-2-2k)d+h)\bigg]\\
& \oplus \bigoplus_{i=0}^{h-1}V((e-1)d+i+1,(m-e+2)d-1) \oplus V((e-1)d+h+1,(m-e+2)d-h-2) \\
&\oplus  \bigoplus_{i=h+1}^{d-2}V((e-1)d+i+1,(m-e+1)d-1)\bigg\} \oplus \bigoplus_{i=2}^{d-1}\bigg\{\bigoplus_{k=0}^{e-3}\bigg[\bigoplus_{l=0}^{h}V(kd+l+i,(e-1+m-2k)d-1) \\
&\oplus \bigoplus_{r=h+1}^{d-1}V(kd+r+i,(e-2+m-2k)d-1)\bigg] \oplus \bigoplus_{l=0}^{h}V((e-2)d+l+i,(m-e+3)d-1) \\
&\oplus \bigoplus_{l=h+1}^{d-1}V((e-2)d+l+i,(m-e+2)d-1)\bigg\} \oplus \bigg\{\bigoplus_{k=0}^{e-3}\bigg[V((k+1)d,(e-2+m-2k)d+h)\\
&\oplus \bigoplus_{i=1}^{h}V((k+1)d+i,(e-2+m-2k)d-1)  \oplus  V((k+1)d+h+1,(e-2+m-2k)d-h-2)    \\
&\oplus \bigoplus_{j=h+2}^{d-1}V((k+1)d+j,(e+m-3-2k)d-1)\bigg] \oplus  V((e-1)d,(m-e+2)d+h)\bigg\}.
\end{split}
\end{equation*}
Now by subtracting the foregoing two identities we get the formula in (1) for $\alpha=ed.$

We are done.
\end{proof}

Next we will determine the ring structure of  $\R(\C_q)$. We will need some preparation of polynomials.

\begin{definition}
The series $\{f_k(x,y,z)\}_{k \ge 0}$  in $\Z[x,y,z]$ is defined inductively as follows:
\begin{itemize}
\item[1)] $f_0(x,y,z)=1,$  $f_1(x,y,z)=y$ and $f_d(x,y,z)=z;$ \\
\item[2)] $f_{md+l}(x,y,z)=yf_{md+l-1}(x,y,z)-xf_{md+l-2}(x,y,z)$ if $m \ge 0$ and $0\leq l \leq d-1;$\\
\item[3)] $f_{(m+1)d}(x,y,z)=zf_{md}(x,y,z)-xf_{(m+1)d-2}(x,y,z)\sum_{i=2}^{d-1}x^if_{md-1}(x,y,z)-x^df_{(m-1)d}(x,y,z).$
\end{itemize}
\end{definition}

\begin{remark}
If $k \leq d-1,$ then $f_k(x,y,z)$ is independent of $z$ and is essentially the generalized Fibonacci polynomials defined in Definition 3.5.
\end{remark}

In order to use $f_k(x,y,z)$ to construct a basis of $\Z[x,y,z],$  we need to define an order of the polynomials as follows.

\begin{definition}
For monomials, define $$\ord(x^iy^lz^m)=(m,l)$$ and we say $(m,l)\geq (m',l')$  if and only if either $m>m'$, or $m=m'$ and $l\geq l'$. Say $(m,l)= (m',l')$  if and only if $m=m'$ and $l=l'.$ Define the order of a polynomial to be the order of the highest order term.
\end{definition}

\begin{lemma}
The highest order term of $f_{md+l}(x,y,z)$ is $y^lz^m$.
\end{lemma}
\begin{proof}
It is easy to see that the highest order term of $f_{l}(x,y,z)$ is $y^l$ if $l\leq d-1$ and that of $f_{d}(x,y,z)$ is $z$. By induction on $m$ one can easily prove the lemma.
\end{proof}

\begin{lemma}
$\big\{x^if_k(x,y,z)|i,k \ge 0\big\}\cup \big\{x^j\big(y^l-\sum_{s=0}^l\binom{l}{s}x^s\big)f_{md-1}(x,y,z)| j \ge 0, \ l,m \ge 1 \big\}$ is a basis of $\Z[x,y,z].$
\end{lemma}
\begin{proof}
Let $g(x,y,z) \in \Z[x,y,z]$ and the highest order term of $g(x,y,z)$ is $h(x)y^lz^m$ where $h(x)\in \Z[x].$
If $l>d-1$, set $$g_1(x,y,z)=g(x)-h(x)\bigg[y^{l-d+1}-\sum_{s=0}^{l-d+1}\binom{l-d+1}{s}x^s\bigg]f_{(m+1)d-1}(x,y,z);$$
and if $l\leq d-1$, set $$g_1(x,y,z)=g(x)-h(x)f_{md+l}(x,y,z). $$
Then the order of $g_1(x,y,z)$ is less than $g(x,y,z).$ Repeat the process for enough times and we will eventually arrive at some $g_k(x,y,z)\in \Z[x].$ This implies that $g(x,y,z)$ is a $\Z$-combination of $\{x^if_k(x,y,z)| i,k \ge 0\}$ and $\big\{x^j\big(y^l-\sum_{s=0}^l\binom{l}{s}x^s\big)f_{md-1}(x,y,z)|j \ge 0,  \ l,m \ge 1 \big\}.$ By considering the order of the highest order term of each polynomial we can get the linear independence.
\end{proof}

\begin{lemma}
$f_{d-1}(x,y,z)$ is a factor of $f_{md-1}(x,y,z)$ for all $m\geq 1.$
\end{lemma}
\begin{proof}
Note first that for any $1\leq i\leq d-1$
\begin{equation}
f_i(x,y,z)f_{md}(x,y,z)=f_{md+i}(x,y,z)+\sum_{j=1}^{i}x^jf_{md-1}(x,y,z).
\end{equation}
Indeed, for $i=1,$
\begin{equation*}
\begin{split}
f_1(x,y,z)f_{md}(x,y,z)&=yf_{md}(x,y,z)\\
&=f_{md+1}(x,y,z)+xf_{md-1}(x,y,z).
\end{split}
\end{equation*}
Assume that
\begin{equation*}
f_i(x,y,z)f_{md}(x,y,z)=f_{md+i}(x,y,z)+\sum_{j=1}^{i}x^jf_{md-1}(x,y,z)
\end{equation*}
for $1\leq i\leq l-1<d-1,$ then we have
\begin{equation*}
\begin{split}
f_l(x,y,z)f_{md}(x,y,z)=&\big[yf_{l-1}(x,y,z)-xf_{l-2}(x,y,z)\big]f_{md}(x,y,z)\\
=&y\bigg[f_{md+l-1}(x,y,z)+\sum_{j=1}^{l-1}x^jf_{md-1}(x,y,z)\bigg]\\
&-x\bigg[f_{md+l-2}(x,y,z)+\sum_{j=1}^{l-2}x^jf_{md-1}(x,y,z)\bigg]\\
=&f_{md+l}(x,y,z)+\sum_{j=1}^{l}x^jf_{md-1}(x,y,z).
\end{split}
\end{equation*}
So by induction we have proved
\begin{equation*}
f_i(x,y,z)f_{md}(x,y,z)=f_{md+i}(x,y,z)+\sum_{j=1}^{i}x^jf_{md-1}(x,y,z).
\end{equation*}
In particular we have
 \begin{equation*}
f_{(m+1)d-1}(x,y,z)=f_{d-1}(x,y,z)f_{md}(x,y,z)-\sum_{j=1}^{d-1}x^jf_{md-1}(x,y,z).
\end{equation*}
This equation leads easily to $f_{d-1}(x,y,z)|f_{(m+1)d-1}(x,y,z)$ with induction on $m.$
\end{proof}
\begin{theorem}
$\R(\C_q)$ is isomorphism to $\Z[x,y,z]/J$ where $J$ is the ideal of $\Z[x,y,z]$ generated by
$$\bigg\{x^n-1, \ (y-x-1)\bigg[\sum_{i=0}^{[\frac{d-1}{2}]}(-1)^i\binom{d-1-i}{i}x^iy^{d-1-2i}\bigg]\bigg\}.$$
\end{theorem}
\begin{proof}
Define
\begin{eqnarray*}
\Phi : \Z[x,y,z] & \longrightarrow & \R(\C_q) \\
         x  & \mapsto & [V(1,0)], \\
         y  & \mapsto & [V(0,1)], \\
         z  & \mapsto & [V(0,d)].
\end{eqnarray*}
By Corollary 3.13, the ring $\R(\C_q)$ is generated by $[V(1,0)], \ [V(0,1)]$ and $ [V(0,d)].$ Hence $\Phi$ is surjective.
From the definition of $f_k(x,y,z)$ and Theorem 3.18 we can see that
\begin{equation*}
\Phi(x^if_k(x,y,z))=[V(1,0)]^{i}f_k([V(1,0)],[V(0,1)],[V(0,d)]) =[V(i,k)].
\end{equation*}
Because
\begin{equation*}
\begin{split}
V(0,1)^{\otimes l}\otimes V(0,md-1)&=\sum_{j=0}^{l}\binom{l}{j}V(j,md-1)\\
&=\sum_{j=0}^{l}\binom{l}{j}V(1,0)^{\otimes j}\otimes V(0,md-1),
\end{split}
\end{equation*}
we have
\begin{equation*}
\Phi\bigg(\bigg[y^l-\sum_{j=0}^l\binom{l}{j}x^j\bigg]f_{md-1}(x,y,z)\bigg)=0.
\end{equation*}
We also have $\Phi (x^n-1)=0$ since $ V(1,0)^{\otimes n}=V(0,0).$ This implies that
\begin{equation*}
J'=\bigg\langle x^n-1, \ \big\{\big[y^l-\sum_{j=0}^l\binom{l}{i}x^j\big]f_{md-1}(x,y,z)|l,m \ge 1 \big\}\bigg\rangle \subseteq \ker \Phi,
\end{equation*}
and so $\Phi$ induces an epimorphism
\begin{eqnarray*}
\overline{\Phi} : \Z[x,y,z]/J' & \longrightarrow & \R(\C_q) \\
         \overline{x}  & \mapsto & [V(1,0)], \\
         \overline{y}  & \mapsto & [V(0,1)], \\
         \overline{z}  & \mapsto & [V(0,d)].
\end{eqnarray*}
Next we prove that $\overline{\Phi}$ is in fact a ring isomorphism. Write $\overline{g(x)}=g(x)+J'\in \Z[x,y,z]/J'$ and by Lemma 3.20 we have
\begin{equation*}
g(x)=\sum c_{i,k}x^if_k(x,y,z) + \sum c_{j,l,m}x^j\big[y^l-\sum_{j=0}^l\binom{l}{j}x^j\big]f_{md-1}(x,y,z),
\end{equation*}
where $c_{i,k}, \ c_{j,l,m}\in \Z.$  So $\overline{g(x)}=\sum c_{i,k}\overline{x^if_k(x,y,z)}.$  It follows that $\{\overline{x^if_k(x,y,z)}|i,k \ge 0\}$ spans  $\Z[x,y,z]/J'.$ Note that $\overline{\Phi}(\overline{x^if_k(x,y,z)})=[V(i,k)]$, and $\{[V(i,k)]|i,k\geq 0\}$ is a basis of $\R(\C_q),$ We can see that $\{\overline{x^if_k(x,y,z)}|i,k \ge 0\}$ is linearly independent, hence a basis of $\Z[x,y,z]/J'.$ Thus $\overline{\Phi}$ is an isomorphism.

Finally we prove that $J'=\big\langle x^n-1,(y-x-1)f_{d-1}(x,y,z)\big\rangle =J.$ It is enough to show that $(y-x-1)f_{d-1}(x,y,z)$ is a factor of $\big[y^l-\sum_{j=0}^l\binom{l}{j}x^j\big]f_{md-1}(x,y,z).$ This follows from Lemma 3.24 and the facts that $y^l-\sum_{j=0}^l\binom{l}{j}x^j=y^l-(1+x)^l$ and $f_{d-1}(x,y,z)=\sum_{i=0}^{[\frac{d-1}{2}]}(-1)^i\binom{d-1-i}{i}x^iy^{d-1-2i}.$

The proof is completed.
\end{proof}

\section{The Green ring of the infinite linear quiver}

\subsection{Hopf structures over the infinite linear quiver}
Let $G=\langle g\rangle$ be the infinite cyclic group and let $\mathcal{A}$
denote the Hopf quiver $Q(G,g).$ Then $\mathcal{A}$ is the
infinite linear quiver. Let $e_i$ denote the arrow $g^i
\longrightarrow g^{i+1}$ and $p_i^l$ the path $e_{i+l-1} \cdots e_i$
of length $l \ge 1,$ for each $i \in \Z.$ The notation
$p_i^0$ is understood as $g^i.$

We collect in this subsection some useful results of graded Hopf
structures on $\k \mathcal{A}.$ The graded Hopf structures are in
one-one correspondence to the left $\k G$-module structures on $\k e_0,$
and thus in one-one correspondence to non-zero elements of $\k.$ Assume $g.e_0=qe_0$ for some $0
\ne q \in \k.$ The corresponding $kG$-Hopf bimodule is $\k e_0 \otimes
\k G.$ We identify $e_i$ and $e_0 \otimes g^i,$ and in this way we
have a $\k G$-Hopf bimodule structure on $\k \mathcal{A}_1.$ We denote
the corresponding graded Hopf algebra by $\k \mathcal{A}(q).$ Recall that the path multiplication formula of $\k \mathcal{A}(q)$ is as follows
\begin{equation}
p_i^l \cdot p_j^m = q^{im} {{l+m}\choose l} _q p_{i+j}^{l+m}.
\end{equation}
In particular, we have
\begin{equation*}
g \cdot p_i^l=q^lp_{i+1}^l, \ \ p_i^l \cdot g=p_{i+1}^l, \ \
a_0^l=l_q!p_0^l \ .
\end{equation*}

The following lemma presents $\k\mathcal{A}(q)$ by generators with relations.
\begin{lemma}\emph{(\cite[Lemma 4.2]{hsaq1}) }
The algebra $\k\mathcal{A}(q)$ can be presented via generators with
relations as follows:
\begin{enumerate}
  \item when $q=1,$ generators: $g, \ g^{-1}, \ e_0.$ relations: $gg^{-1}=1=g^{-1}g, \ ge_0=e_0g.$
  \item when $q \ne 1$ is not a root of unity, generators: $g, \ g^{-1}, \ e_0.$
        relations: $gg^{-1}=1=g^{-1}g, \ ge_0=qe_0g.$
  \item when $q \ne 1$ is a root of unity of order $d,$ generators: $g, \ g^{-1}, \ e_0, \
        p_0^d.$ relations: $gg^{-1}=1=g^{-1}g, \ e_0^d=0, \ ge_0=qe_0g, \ gp_0^d=p_0^dg, \ e_0p_0^d=p_0^de_0.$
\end{enumerate}
\end{lemma}

\subsection{The tensor category associated to $\k \mathcal{A}(q)$}
For each $i\in \Z$  and $ l\geq 0$, let $V(i,l)$ be a linear space with $\k$-basis $\{v^i_j\}_{0\leq j\leq l}.$ We endow on $V(i,l)$ a $\mathcal{A}$-representation structure by letting $$V(i,l)_j=\left\{
                                                              \begin{array}{ll}
                                                                \k v^i_{j-i}, & \hbox{$i \le j \le i+l;$} \\
                                                                0, & \hbox{otherwise.}
                                                              \end{array}
                                                            \right.$$
and letting $V(i,l)_{e_j}$ maps $v_{j-i}^i$ to $v_{j-i+1}^i$ for all $i \le j \le i+l.$ Here by convention $v^i_k$ is understood as $0$ if $k>l.$ Note also that $V(i,l)$ can be viewed as a right $\k \mathcal{A}(q)$-comodule via
\begin{equation*}
\begin{split}
\d :V(i,l) \longrightarrow  V(i,l)\otimes \k\mathcal{A}(q)\\
v^i_m \mapsto \sum^l_{j=m}v^i_j\otimes p_{i+m}^{j-m}.
\end{split}
\end{equation*}
By (4.1), the comodule structure of $V(i,l)\otimes V(j,m)$ is given by
\begin{equation*}
\d(v^i_s\otimes v^j_t)=\sum^l_{x=s}\sum^m_{y=t}q^{(i+s)(y-t)}\binom{x+y-s-t}{x-s}_qv^i_x\otimes v^j_y\otimes p_{i+j+s+t}^{x+y-s-t}.
\end{equation*}
It is well known that $\{V(i,l)|i\in \Z, \ l \ge 0\}$ is a complete set of finite dimensional indecomposable representations of the quiver $\mathcal{A},$ and thus a complete set of finite dimensional indecomposable comodules of $\k\mathcal{A}(q).$  Similar to subsection 3.2, we will also view $V(i,l)\otimes V(j,m)$ as a rational module of $(\k\mathcal{A}(q))^*,$ the dual algebra of $\k\mathcal{A}(q).$ The module structure map is given by
\begin{equation*}
  (p_e^f)^*(v^i_s\otimes v^j_t)=\sum_{x=0}^fq^{(i+s)(f-x)}\binom{f}{x}_qv^i_{s+x}\otimes v^j_{t+f-x}.
\end{equation*}
In the section, we let $\md_q$ denote the tensor category of finite dimensional $\k \mathcal{A}(q)$-comodules and $\R(\md_q)$ the Green ring of $\md_q$.

Recall that Lemma 3.10 plays an important role in the computation of the Clebsh-Gordan formulae of $\C_q.$  Here for $\md_q$ we have a similar lemma. As the proof is similar, we do not repeat it.
\begin{lemma}
Assume $V=\oplus_{s,t}V(s,t)$ and $V(i,j) \subset V$ in $\md_q$. Then there exists an inclusion map $\phi:V(i,j) \To V(s,t)$ for some indecomposable direct summand $V(s,t)$ of $V$ with $t\geq j$ and $i+j=s+t$.
\end{lemma}

\subsection{The case of $q=1$ or $q$ being not a root of unity} In this subsection we compute the Clebsch-Gordan formulae and Green ring of $\md_q$ if $q=1$ or $q$ is not a root of unity.

Similar to Section 3, we have the following
\begin{lemma}
\begin{eqnarray*}
&V(0,0)\otimes V(i,l)=V(i,l)\otimes V(0,0)=V(i,l),\\
&V(1,0)\otimes V(i,l)=V(i,l)\otimes V(1,0)=V(i+1,l),\\
&V(-1,0)\otimes V(i,l)=V(i,l)\otimes V(-1,0)=V(i-1,l),\\
&V(1,0)^{\otimes m}=V(m,0), \quad V(-1,0)^{\otimes m}=V(-m,0),\\
&V(m,0)\otimes V(-m,0)=V(-m,0)\otimes V(m,0)=V(0,0).
\end{eqnarray*}
\end{lemma}

\begin{lemma}
\begin{equation}
V(0,1)\otimes V(0,l)=V(0,l+1)\oplus V(1,l-1)= V(0,l)\otimes V(0,1).
\end{equation}
\end{lemma}
\begin{proof} As before, we only prove the first equality. Consider the maps
\begin{eqnarray*}
\phi_1 :V(0,l+1) & \longrightarrow  & V(0,1)\otimes V(0,l) \\
    v_0^0        & \mapsto          & v_0^0\otimes v_0^0,\\
    v_i^0        & \mapsto          & v_0^0\otimes v_i^0 +(i)_q v_1^0\otimes v_{i-1}^0, \quad i =1,2, \cdots, l,\\
    v_{l+1}^0    & \mapsto          & (l+1)_q v_1^0\otimes v_l^0
\end{eqnarray*}
and
\begin{eqnarray*}
\phi_2 :V(1,l-1) & \longrightarrow  & V(0,1)\otimes V(0,l) \\
    v_i^1        & \mapsto          & v_0^0\otimes v_{i+1}^0 +q^{-1}(l-i)_{q^{-1}} v_1^0\otimes v_{i}^0, \quad i=1,2, \cdots, l-1.
\end{eqnarray*}
It is not difficult to verify that $\phi_1$ and $\phi_2$ are comodule monomorphisms. Then by Lemma 4.2 we have $$ V(0,l+1)\oplus V(1,l-1)\subset V(0,1)\otimes V(0,l).$$ Now the claimed equality follows by comparing the dimensions.
\end{proof}

With a help of Lemma 4.4, one may prove the following identity easily by induction on $l.$
\begin{proposition}
$V(0,l)\otimes V(0,m)= \bigoplus_{i=0}^l V(i,l+m-2i)=V(0,m)\otimes V(0,l).$
\end{proposition}

Combining Lemma 4.2 and Proposition 4.4, we get the Clebsch-Gordan formulae for $\md_q$ as follows.
\begin{corollary}
$V(s,l)\otimes V(t,m)= \bigoplus_{i=0}^l V(s+t+i,l+m-2i)=V(t,m)\otimes V(s,l).$
\end{corollary}

Now we are ready to give the Green ring structure of $\md_q$ when $q=1$ or $q$ is not a root of unit. As before, we need some facts about polynomials. Let $f_k(x,y)$ be the generalized Fibonacci polynomial as defined in Definition 3.5. The we have the following
\begin{lemma}
$\{x^if_k(x,y)\}_{i\in \Z, \ k \ge 0}$ is a basis of $\Z[x,x^{-1},y]$.
\end{lemma}
\begin{proof} Similar to Lemma 3.6. \end{proof}

\begin{theorem}
$\R(\md_q)$ is isomorphic to $\Z[x,x^{-1},y]$ when $q=1$ or $q$ is not a root of unit.
\end{theorem}
\begin{proof}
Define a ring  morphism
\begin{eqnarray*}
\Phi :\Z[x,x^{-1},y] & \longrightarrow  & \R(\md_q) \\
    x        & \mapsto          & [V(1,0)],\\
    x^{-1}        & \mapsto          & [V(-1,0)],\\
    y    & \mapsto          & [V(0,1)].
\end{eqnarray*}
By the definition of $f_k(x,y)$ and (4.2), it is not hard to verify that $$\Phi(x^if_k(x,y))=[V(i,k)].$$ This says that $\Phi$ maps the basis $\{x^if_k{x,y}\}_{i\in \Z, \ k \ge 0}$ of $\Z[x,x^{-1},y]$ to the basis $\{[V(i,k)]\}_{i\in \Z, \ k \ge 0}$ of $\R(\md_q),$ thus it is an isomorphism.
\end{proof}

\subsection{The case of $q$ being a root of unity} In this subsection $q$ is assumed to be a root of unity with multiplicative order $d \ge 2.$ First we list some facts without proofs, as they can be given in a similar way as before.

\begin{lemma}
\begin{eqnarray*}
&V(0,0)\otimes V(i,l)=V(i,l)\otimes V(0,0)=V(i,l),\\
&V(1,0)\otimes V(i,l)=V(i,l)\otimes V(1,0)=V(i+1,l),\\
&V(-1,0)\otimes V(i,l)=V(i,l)\otimes V(-1,0)=V(i-1,l).
\end{eqnarray*}
\end{lemma}

\begin{lemma}
\begin{equation*}
\begin{split}
&V(0,1)\otimes V(0,md+l)=V(0,md+l)\otimes V(0,1)\\
=&\begin{cases} V(0,md+l+1)\oplus V(1,md+l-1), & 0\leq l\leq d-2; \\ V(0,(m+1)d-1)\oplus V(1,(m+1)d-1),  &
l=d-1.
\end{cases}
\end{split}
\end{equation*}
\end{lemma}
\begin{lemma} For all $m\geq 1$ one has
\begin{equation*}
\begin{split}
&V(0,d)\otimes V(0,md)=V(0,md)\otimes V(0,d)\\
=&V(0,(m+1)d)\oplus V(1,(m+1)d-2)\oplus \bigoplus_{i=2}^{d-1}V(i,md-1)\oplus V(d,(m-1)d).
\end{split}
\end{equation*}
\end{lemma}

Obviously, Lemmas 4.9-4.11 imply that the Green ring $\R(\md_q)$ is a commutative ring and is generated by $[V(-1,0)], \ [V(1,0)], \ [V(0,1)]$ and $[V(0,d)].$ The preceding three Lemmas also help to compute the decomposition of $V(0,ed+f)\otimes V(0,md+h)$ for all $e, m\geq 0$ and $0\leq f,h\leq d-1.$ In the following let $\gamma=f+h-d+1.$

\begin{theorem}
\begin{itemize}
\item[(1)] If $0\leq f\leq h$ and $f+h<d-1,$ then
\begin{equation*}
\begin{split}
&V(0,ed+f)\otimes V(0,md+h)=\bigoplus_{k=0}^{e-1}\bigg[\bigoplus_{i=0}^{f}V(kd +i,(e+m-2k)d+f+h-2i) \\
&\quad \oplus \bigoplus_{j=f+1}^{h}V(kd+j,(e+m-2k)d-1) \oplus   \bigoplus_{r=h+1}^{f+h+1}V(kd+r,(e+m-2k)d+h+f-2r) \\  &\quad \oplus \bigoplus_{s=f+h+2}^{d-1}V(kd+s,(e+m-1-2k)d-1)\bigg] \oplus \bigoplus_{i=0}^{f}V(ed+i,(m-e)d+h+f-2i);
\end{split}
\end{equation*}
\item[(2)]if $0<f\leq h<d$ and $f+h\geq d,$ then
\begin{equation*}
\begin{split}
&V(0,ed+f)\otimes V(0,md+h)=\bigoplus_{k=0}^{e-1}\bigg[\bigoplus_{i=0}^{\gamma}V(kd+i,(e+m-2k)d+d-1) \\
&\quad \oplus \bigoplus_{j=\gamma+1}^{f}V(kd+j,(e+m-2k)d+f+h-2i) \oplus   \bigoplus_{r=f+1}^{h}V(kd+r,(e+m-2k)d-1)\\  &\quad \oplus \bigoplus_{s=h+1}^{d-1}V(kd+s,(e+m-2k)d+h+f-2s)\bigg] \oplus \bigoplus_{i=0}^{\gamma}V(ed+i,(m-e)d+d-1) \\
&\quad \oplus \bigoplus_{i=\gamma+1}^{f}V(ed+i,(m-e)d+f+h-2i);
\end{split}
\end{equation*}
\item[(3)]if $0\leq h<f<d$ and $f+h < d-1,$ then
\begin{equation*}
\begin{split}
&V(0,ed+f)\otimes V(0,md+h) \\
=& \bigoplus_{k=0}^{e-1}\bigg[\bigoplus_{i=0}^{h}V(kd+i,(e+m-2k)d+h+f-2i)\oplus\bigoplus_{j=h+1}^{f}V(kd+j,(e+m-2k)d-1)\\
&\oplus\bigoplus_{r=f+1}^{h+f+1}V(kd+r,(e+m-2k)d+h+f-2r)\oplus \bigoplus_{s=f+h+2}^{d-1}V(kd+s,(e+m-1-2k)d-1)\bigg] \\
&\oplus \bigoplus_{i=0}^{h}V(ed+i,(m-e)d+h+f-2i)\oplus \bigoplus_{i=h+1}^{f}V(ed+i,(m-e)d-1);
\end{split}
\end{equation*}
\item[(4)]if $0<h<f<d$ and $f+h\geq d,$ then
\begin{equation*}
\begin{split}
&V(0,ed+f)\otimes V(0,md+h)=\bigoplus_{k=0}^{e-1}\bigg[\bigoplus_{i=0}^{\gamma}V(kd+i,(e+m-2k)d+d-1) \\
&\quad \oplus \bigoplus_{j=\gamma+1}^{h}V(kd+j,(e+m-2k)d+f+h-2j) \oplus \bigoplus_{r=h+1}^{f}V(kd+r,(e+m-2k)d-1) \\
&\quad \oplus \bigoplus_{s=f+1}^{d-1}V(kd+s,(e+m-2k)d+h+f-2s)\bigg] \oplus \bigoplus_{i=0}^{\gamma}V(ed+i,(m-e)d+d-1) \\
&\quad \oplus \bigoplus_{i=\gamma+1}^{h}V(ed+i,(m-e)d+f+h-2i) \oplus  \bigoplus_{i=h+1}^{f}V(ed+i,(m-e)d-1).
\end{split}
\end{equation*}
\end{itemize}
\end{theorem}
\begin{proof} Similar to Theorem 3.18. \end{proof}

Let $\{f_k(x,y,z)\}_{k \ge 0}$ be the polynomials as defined in Definition 3.19 and we have the following easy fact.
\begin{lemma}
$\big\{x^if_k(x,y,z)|i\in \Z, \ k \ge 0 \big\} \bigcup \big\{x^j(y^l-\sum_{s=0}^l\binom{l}{s}x^sf_{md-1}(x,y,z)|j\in \Z, \ l,m \ge 1 \big\}$ is a basis of $\Z[x,x^{-1},y,z].$
\end{lemma}

Now we are in the position to determine the Green ring $\R(\md_q).$
\begin{theorem}
$\R(\md_q)$ is isomorphic to $\Z[x,x^{-1},y,z]/J$ where $J$ is the ideal of $\Z[x,x^{-1},y,z]$ generated by $(y-x-1)\bigg[\sum_{i=0}^{[\frac{d-1}{2}]}(-1)^i\binom{d-1-i}{i}x^iy^{d-1-2i}\bigg].$
\end{theorem}

\begin{proof} The proof is similar to that of Theorem 3.25.
Define a ring map
\begin{eqnarray*}
\Psi : \Z[x,x^{-1},y,z] & \longrightarrow & \R(\md_q) \\
         x  & \mapsto & [V(1,0)], \\
         x^{-1} & \mapsto & [V(-1,0)],\\
         y  & \mapsto & [V(0,1)], \\
         z  & \mapsto & [V(0,d)].
\end{eqnarray*}
Obviously $\Psi$ is surjective as $x,x^{-1},y,z$ map to a set of generators of $\R(\md_q).$
According to the definition of $f_k(x,y,z)$ and Lemmas 4.10-4.11, we have
\begin{eqnarray*}
\Psi(x^if_k(x,y,z))=[V(1,0)]^{i}f_k([V(1,0)],[V(0,1)],[V(0,d)]) =[V(i,k)]\\
\Psi(x^{-i}f_k(x,y,z))=[V(-1,0)]^{i}f_k([V(1,0)],[V(0,1)],[V(0,d)]) =[V(-i,k)]
\end{eqnarray*} for all $i \ge 0.$
Because
\begin{equation*}
\begin{split}
V(0,1)^{\otimes l}\otimes V(0,md-1)&=\sum_{j=0}^{l}\binom{l}{j}V(j,md-1)\\
&=\sum_{j=0}^{l}\binom{l}{j}V(1,0)^{\otimes j} \otimes V(0,md-1),
\end{split}
\end{equation*}
we have
\begin{equation*}
\Psi\bigg(\big[y^l-\sum_{j=0}^l\binom{l}{j}x^j\big]f_{md-1}(x,y,z)\bigg)=0.
\end{equation*}
This implies that
\begin{equation*}
J'=\bigg\langle \bigg\{\big[y^l-\sum_{j=1}^l\binom{l}{i}x^j\big]f_{md-1}(x,y,z)\bigg\}_{l,m \ge 1}\bigg\rangle \subseteq \ker\Psi.
\end{equation*}
Thus $\Psi$ induces an epimorphism
\begin{eqnarray*}
\overline{\Psi} : \Z[x,x^{-1},y,z]/J' & \longrightarrow & \R(\md_q) \\
         \overline{x}  & \mapsto & [V(1,0)], \\
         \overline{x^{-1}}  & \mapsto & [V(-1,0)], \\
         \overline{y}  & \mapsto & [V(0,1)], \\
         \overline{z}  & \mapsto & [V(0,d)].
\end{eqnarray*}
Next we prove that $\overline{\Psi}$ is a ring isomorphism. Let $\overline{g(x,x^{-1},y,z)}=g(x,x^{-1},y,z)+J' \in \Z[x,x^{-1},y,z]/J'.$ By Lemma 4.11 we have
\begin{equation*}
g(x)=\sum_{i\in \Z, \ k \ge 0} c_{i,k}x^if_k(x,y,z) + \sum_{j\in \Z, \ m \ge 0, \ l \ge 1} c_{j,l,m}x^j\big[y^l-\sum_{j=1}^l\binom{l}{i}x^j\big]f_{md-1}(x,y,z),
\end{equation*}
where $c_{i,k},c_{j,l,m}\in \Z.$
So $$\overline{g(x,x^{-1},y,z)}=\sum_{i\in \Z, \ k \ge 0} c_{i,k}\overline{x^if_k(x,y,z)}.$$ In other words, $\{\overline{x^if_k(x,y,z)}|i\in \Z, \ k\geq 0\}$ spans $\Z[x,x^{-1},y,z]/J'.$ Note that $\overline{\Psi}(\overline{x^if_k(x,y,z)})=[V(i,k)]$, and $\{[V(i,k)]|i\in \Z, \ k\geq 0\}$ is a basis of $\R(\md_q).$ It follows that $\{\overline{x^if_k(x,y,z)}|i\in \Z, \ k\geq 0\}$ is linearly independent, hence a basis of $\Z[x,x^{-1},y,z]/J'.$  So $\overline{\Psi}$ is an isomorphism. Now by Lemma 3.24 we have $J'=\langle (y-x-1)f_{d-1}(x,y,z)\rangle=J$.

We are done.
\end{proof}

%

\end{document}